\documentclass[11pt]{amsart}
\usepackage[margin=25mm]{geometry}
\usepackage{amssymb,amsmath,amsthm,mathtools}
\usepackage{enumitem,microtype,xcolor}
\definecolor{darkblue}{rgb}{0,0,0.6}
\usepackage[ocgcolorlinks]{hyperref}
\usepackage[capitalize]{cleveref}
\hypersetup{colorlinks=true,citecolor=darkblue,linkcolor=darkblue,urlcolor=darkblue,breaklinks,pdftitle={The relation gap problem and Wall's D2 problem for virtually free groups},pdfauthor={John Nicholson}}

\newcommand{\Z}{\mathbb Z}

\newcommand{\C}{\mathbb C}
\newcommand{\Dtwo}{\textnormal{D2}}

\newcommand{\ab}{\mathrm{ab}}

\DeclareMathOperator{\deficiency}{def}
\newcommand{\ncl}[1]{\langle\!\langle #1\rangle\!\rangle}
\DeclareMathOperator{\im}{im}

\newtheorem{theorem}{Theorem}

\newtheorem{corollary}[theorem]{Corollary}
\theoremstyle{remark}
\newtheorem*{remark}{Remark}

\title{Two-generator presentations with a relation gap}

\author{Laurent Bartholdi}
\address{Institut Camille Jordan, Universit\'{e} de Lyon 1 \textnormal{\itshape and} Section de Math\'{e}matiques, Universit\'{e} de Gen\`{e}ve}
\email{laurent.bartholdi@gmail.com}
\author{Martin R. Bridson}
\address{Mathematical Institute, University of Oxford, Andrew Wiles Building, Oxford OX2 6GG, UK}
\email{bridson@maths.ox.ac.uk}
\author{Roman Mikhailov}
\address{Saint Petersburg State University, 7/9 Universitetskaya nab., St. Petersburg, 199034, Russia}
\author{John Nicholson}
\address{Max Planck Institute for Mathematics, Vivatsgasse 7, 53111 Bonn, Germany}
\email{math@johnnynicholson.co.uk}


\begin{document}
\begin{abstract}
We exhibit two-generator two-relator group presentations that have cyclic relation modules but do not present one-relator groups; this resolves the relation gap problem.  
By attaching a single $3$-cell to each of the corresponding presentation 2-complexes, one obtains counterexamples to Wall's D2 problem.
\end{abstract}
\maketitle

Let $F$ be a finitely generated free group, suppose that $R\unlhd F$ is the normal closure of a finite set, and let $G=F/R$. The \emph{relation module} $R_{\ab}=R/[R,R]$ is a left $\Z G$-module under conjugation. Write $d_F(R)$ for the minimum number of normal generators of $R$ in $F$, and $d_{\Z G}(R_{\ab})$ for the minimum number of module generators. The \emph{relation gap problem} asks whether there exists an example with
\[
d_F(R)>d_{\Z G}(R_{\ab}).
\]
This problem originated in the work of K.W. Gruenberg \cite{Gruenberg76} and appeared as Problem B15 in \cite{Wa79}.
The case $d_{\Z G}(R_{\ab}) = 1$ has attracted particular interest: one wants to know if a group with cyclic relation module is a one-relator group \cite[Problem OR5]{BMS99}. For background, see \cite{Harlander18}.

The following theorem settles the relation gap problem. Our convention is $[u,v]:=uvu^{-1}v^{-1}$.

\begin{theorem}\label{thm:gap}
Let $p, q \ge 2$ be coprime integers, let $F=F(x,t)$ and let $R$ be the normal closure of $x^{pq}$ and $[x^p,tx^qt^{-1}]$. Then
\[
G=F/R=\langle x,t\mid x^{pq},[x^p,tx^qt^{-1}]\rangle
\]
is not a one-relator group, but $R_{\ab}$ is cyclic. Therefore, $d_F(R)=2$ and $d_{\Z G}(R_{\ab})=1$.
\end{theorem}

\begin{proof}
Let $a=x^p$ and $b=tx^qt^{-1}$, which satisfy $a^q=b^p=1$ and $ab=ba$. The epimorphism $\phi\colon G 
\twoheadrightarrow C_{pq} =\langle c \rangle$ sending $x \mapsto c$, $t \mapsto 1$ shows that $a$ has order $q$.
Both $x^q$ and $b$ centralise $a$, but their images do not commute in
$
G/\ncl{x^p}\cong C_p*\Z,
$
by the normal form for a free product. Hence $C_G(a)$ is nonabelian. Newman's theorem \cite[\S3]{Newman73} states that every nonidentity element of a one-relator group with torsion has cyclic centraliser. Thus $G$ is not one-relator, and $d_F(R)=2$.

To see that the relation module is cyclic, we work in the universal cover of the presentation 2-complex $Y$.  The 1-skeleton of $\widetilde{Y}$ is the Cayley graph
${\rm{Cay}}(G, \{x,t\})$ and the first homology of this graph, with the action of $G=\pi_1Y$, is the relation module. But this homology is the image
of the second boundary map $\partial_2\colon \Z G^2 \to \Z G^2$ of $\widetilde{Y}$. Let $N_x=1+x+\cdots+x^{pq-1}$.
With rows indexed by the relators and columns indexed by $x,t$, the boundary map is given by the Fox matrix \cite{Fox53}:
\begin{equation}\label{eq:fox}
\partial_2 = \begin{pmatrix}N_x&0\\ u&(a-1)(1-b)\end{pmatrix},
\quad \text{where} \,\, u=(1-b)\sum_{i=0}^{p-1}x^i+(a-1)t\sum_{i=0}^{q-1}x^i.
\end{equation} 
Each $(r,s)\in\im \partial_2$ satisfies $r(x-1)+s(t-1)=0$. 
The epimorphism $G\twoheadrightarrow\Z$ that kills $x$ shows that $t$ has infinite order in $G$.
Thus right multiplication by $t-1$ is injective on $\Z G$: an element fixed by right multiplication by $t$ has constant coefficients on the infinite orbits, and hence cannot have finite support unless it is zero. Projection onto the first coordinate therefore gives us an identification
\begin{equation}\label{eq:ideal}
R_{\ab}\cong I:=\Z G \cdot N_x+\Z G \cdot u.
\end{equation}
It now suffices to prove that the module $I$ is generated by a single element.

To clarify the exposition, we restrict for a moment to the case $p=3$ and $q=2$. Set $N_a=1+a$ and $N_b=1+b+b^2$. We have
\begin{equation}\label{eq:identities}
N_aN_x=2N_x,\qquad N_au=(1-b)N_x,\qquad N_bu=(a-1)tN_x, 
\end{equation}
 because $N_a(a-1)=N_b(1-b)=0$, $N_a(1+x+x^2)=N_x$ and $N_bt=t(1+x^2+x^4)$. The generator of $I$ that we seek is
\begin{equation}\label{eq:generator}
v:=N_x+(b+b^2)u=\bigl(1+(a-1)t\bigr)N_x-u.
\end{equation}
The two expressions for $v$ agree because $N_b-1=b+b^2$ and $N_bu=(a-1)tN_x$. The second expression tells us that
$I=\Z G \cdot N_x+\Z G \cdot v$. To see that $N_x\in\Z G \cdot v$, we use the identities
\[
N_av=(1+b)N_x,\qquad N_bv=N_bN_x+2(a-1)tN_x.
\]
Since $b^3=1$, we obtain
\[
(1-b+b^2)N_av=2N_x,\qquad (N_b-bN_a)v=N_x+2(a-1)tN_x.
\]
Multiplying the first equality by  $(a-1)t$ and subtracting from the second gives
\[
N_x=\bigl[N_b-bN_a-(a-1)t(1-b+b^2)N_a\bigr]v.
\]
Thus $I=\Z G \cdot v$, as claimed. 

The general case $p, q \ge 2$ is similar but notationally challenging. The generator is
$v = N_x - \Delta(1+\ell N_b)u$,
where $N_b = \sum_{i=0}^{p-1} b^i$, $\Delta = \sum_{i,j \ge 0, i+j \le q-2} a^ib^j$, and $\ell\in\mathbb{N}$ with $\ell p \equiv -1$ mod $q$.
\end{proof}

The examples in Theorem \ref{thm:gap} are related to
OpenAI's proposed counterexample to Wall's D2 problem \cite{OpenAI26}. However, our proof is both elementary and completely independent of that work. 

\begin{remark} (i) The epimorphism $\phi\colon G\twoheadrightarrow C_{pq} = \langle c \rangle$ sending $x \mapsto c$, $t \mapsto 1$ can be split by $c \mapsto x$.
It can be shown using Reidemeister–Schreier rewriting that $\ker(\phi) \cong F_{p+q-1}$. Hence $G\cong F_{p+q-1}\rtimes C_{pq}$.

(ii) For $p=3$ and $q=2$, let $H = \langle x, t \mid x^{6}[a, b^3][a,b]^{-1}\rangle$ where $a = x^3$ and $b = tx^2t^{-1}$. It can be shown that $G$ is the maximal residually solvable quotient of $H$. Since $G$ is not a one-relator group, this gives a counterexample to a recent conjecture of Linton \cite[Conjecture 4.8]{Li26}.
\end{remark}

A finite connected CW complex $X$ of dimension at most three is called a \emph{D2 complex} if $H_3(\widetilde X;\Z)=0$ and $H^3(X;M)=0$ for every left $\Z\pi_1(X)$-module $M$. This condition is equivalent to the cellular boundary map $\partial_3:C_3(\widetilde X)\to C_2(\widetilde X)$ being split injective. The \emph{D2 problem} of C.T.C.~Wall asks whether every such complex is homotopy equivalent to a finite $2$-complex. This dates back to Wall's 1965 paper on finiteness conditions \cite[\S2]{Wall65} and appeared as Problem D3 in \cite{Wa79}.

Unlike Theorem \ref{thm:gap}, the following result relies on a recent claim from OpenAI, specifically  
\cite[Corollary 2.3]{OpenAI26}, and the reader can regard Corollary~\ref{thm:d2} as conditional upon this statement.
Our construction of a D2 complex follows an idea of M.N.~Dyer described in \cite[Theorem 3.5]{Ha00}.

\begin{corollary}\label{thm:d2}
Let $G=F/R$ be as in Theorem \ref{thm:gap}, with presentation $2$-complex $Y$.
Then one can attach a single $3$-cell to $Y$ to obtain a finite \Dtwo{} complex $X$ with $\pi_1(X)=G$ and $\chi(X)=0$. 
In contrast, every finite connected $2$-complex $Z$ with $\pi_1(Z) = G$ has $\chi(Z) \ge 1$. Therefore, $X$ is not homotopy equivalent to a finite $2$-complex.
\end{corollary}

Our example is optimal in the sense that it has six cells, which is the minimum possible number for a counterexample to Wall's D2 problem. The Euler characteristic bound gives $\deficiency(G) = 0$ which strengthens, by less elementary means, the earlier result that $G$ is not a one-relator group.

\begin{proof}
We retain the notation from the proof of Theorem \ref{thm:gap}. Let $e_1,e_2$ be the cellular basis of $C_2(\widetilde Y)$ and put $e'_1=e_1 - \Delta (1+\ell N_b)e_2$. 
As in the proof of \cref{thm:gap}, $\partial_2e'_1$ generates $\im\partial_2$. Choose $\lambda\in\Z G$ with $\partial_2e_2=\lambda\partial_2e'_1$, and put $e'_2=e_2-\lambda e'_1$. Then $e'_2$ is a cycle and $(e'_1,e'_2)$ is a basis, obtained from $(e_1,e_2)$ by two elementary changes of basis. We have $\ker(\partial_2) \cong H_2(\widetilde Y) \cong \pi_2(Y)$, so there is a map from the 2-sphere to $\widetilde Y$ representing the homology class of $e'_2$. We compose this map with the covering 
projection $\widetilde Y\to Y$ and use it to attach a $3$-cell to $Y$; call the resulting 3-complex $X$. By construction, 
$X$ is a D2 complex, with $\pi_1(X)=G$ and $\chi(X)=1-2+2-1=0$.

To bound the Euler characteristic of finite $2$-complexes 
with fundamental group $G$, we consider the action of $G$ on $\mathbb{C}$ where $x\in G$ acts as multiplication by a 
primitive $pq$-th root of unity $\zeta$  and $t$ acts trivially. Let $\C_\rho$ be the corresponding right $\Z G$-module. We first observe that $H_0(Y;\C_\rho)=H_1(Y;\C_\rho)=0$. This follows by applying $\C_\rho \otimes_{\Z G} -$ to $C_*(\widetilde Y)$, using \eqref{eq:fox} to obtain
\[
\partial_2^\rho=\begin{pmatrix}0&0\\0&(\zeta^p - 1)(1 - \zeta^q)\end{pmatrix},
\qquad \partial_1^\rho=\begin{pmatrix}\zeta - 1\\0\end{pmatrix}.
\]

Let $Z$ be a finite connected $2$-complex with $\pi_1(Z) = G$. Since $Y$ and $Z$ can each serve as the 2-skeleton of a $K(G,1)$, we have $H_0(Z;\C_\rho)=H_1(Z;\C_\rho)=0$. Since $x \in G$ has finite order and $\rho(x) \ne 1$, from \cite[Corollary 2.3]{OpenAI26} we have $H_2(Z;\C_\rho) \ne 0$. The Euler characteristic can be
calculated from twisted homology (see \cite[Proposition 2.5.4(ii)]{Di04} for a precise statement). In our setting this yields $\chi(Z) = \sum_{i=0}^2 (-1)^i \dim_{\C}H_i(Z;\C_\rho) = \dim_{\C}H_2(Z;\C_\rho) \ge 1$.
\end{proof}

\subsection*{Acknowledgements}
LB is partly supported by the SNSF grant TMAG-2\_216487/1.
JN is supported by the Max Planck Institute for Mathematics. 

\subsection*{Declaration of AI usage}  ChatGPT 6 Astra was used to assist with the discovery of the groups in Theorem A and the proof of that result, as well as the preparation of an initial draft manuscript. Claude Fable 5.1 was consulted in subsequent editing. The arguments and text have been thoroughly checked and substantially revised by the authors, who are responsible for the final manuscript.  

\bibliographystyle{amsalpha}
\bibliography{biblio}

\providecommand{\bysame}{\leavevmode\hbox to3em{\hrulefill}\thinspace}
\providecommand{\MR}{\relax\ifhmode\unskip\space\fi MR }
\providecommand{\MRhref}[2]{%
  \href{http://www.ams.org/mathscinet-getitem?mr=#1}{#2}
}
\providecommand{\href}[2]{#2}
\begin{thebibliography}{{Ope}26}

\bibitem[BMS99]{BMS99}
Gilbert Baumslag, Alexei~G. Myasnikov, and Vladimir Shpilrain, \emph{Open problems in combinatorial group theory}, Groups, languages and geometry ({S}outh {H}adley, {MA}, 1998), Contemp. Math., vol. 250, Amer. Math. Soc., Providence, RI, 1999, pp.~1--27.

\bibitem[Dim04]{Di04}
Alexandru Dimca, \emph{Sheaves in topology}, Universitext, Springer-Verlag, Berlin, 2004.

\bibitem[Fox53]{Fox53}
Ralph~H. Fox, \emph{Free differential calculus. {I}. {Derivation} in the free group ring}, Ann. of Math. (2) \textbf{57} (1953), no.~3, 547--560.

\bibitem[Gru76]{Gruenberg76}
Karl~W. Gruenberg, \emph{Relation modules of finite groups}, CBMS Regional Conference Series in Mathematics, vol.~25, American Mathematical Society, Providence, RI, 1976.

\bibitem[Har00]{Ha00}
Jens Harlander, \emph{Some aspects of efficiency}, Groups---{K}orea '98 ({P}usan), de Gruyter, Berlin, 2000, pp.~165--180.

\bibitem[Har18]{Harlander18}
\bysame, \emph{A survey of recent progress on some problems in 2-dimensional topology}, Advances in two-dimensional homotopy and combinatorial group theory (W.~Metzler and S.~Rosebrock, eds.), London Mathematical Society Lecture Note Series, vol. 446, Cambridge University Press, Cambridge, 2018, pp.~1--26.

\bibitem[Lin26]{Li26}
Marco Linton, \emph{Residually rationally solvable one-relator groups}, Bull. Lond. Math. Soc. \textbf{58} (2026), no.~1, Paper No. e70210, 12.

\bibitem[New73]{Newman73}
B.~B. Newman, \emph{The soluble subgroups of a one-relator group with torsion}, J. Austral. Math. Soc. \textbf{16} (1973), 278--285.

\bibitem[{Ope}26]{OpenAI26}
{OpenAI}, \emph{A counterexample to {Wall}'s {$D(2)$} problem}, \href{https://github.com/openai/math/blob/main/preprints/A-Counterexample-to-Walls-D2-Problem-October-6-2026/wall-d2-counterexample.pdf}{OpenAI Math Release preprint}, 2026, 6 October 2026.

\bibitem[Wal65]{Wall65}
C.~T.~C. Wall, \emph{Finiteness conditions for {CW}-complexes}, Ann. of Math. (2) \textbf{81} (1965), no.~1, 56--69.

\bibitem[Wal79]{Wa79}
\bysame, \emph{List of problems}, Homological group theory ({P}roc. {S}ympos., {D}urham, 1977), London Math. Soc. Lecture Note Ser., vol.~36, Cambridge Univ. Press, Cambridge-New York, 1979, pp.~369--394.

\end{thebibliography}
\end{document}